\documentclass[11pt, letterpaper,leqno]{amsart}

\usepackage{amsmath,amsthm,amsrefs,amssymb,hyperref}
\usepackage[usenames,dvipsnames]{xcolor}
\usepackage{float}
\usepackage{todonotes}
\usepackage{fancyhdr}

\numberwithin{equation}{section}

\newtheorem{lemma}{Lemma}[section]
\newtheorem{theorem}{Theorem}[section]

\newtheorem{proposition}{Proposition}[section]

\theoremstyle{definition}

\newtheorem{innercustomgeneric}{\customgenericname}
\providecommand{\customgenericname}{}
\newcommand{\newcustomtheorem}[2]{%
	\newenvironment{#1}[1]
	{%
		\renewcommand\customgenericname{#2}%
		\renewcommand\theinnercustomgeneric{##1}%
		\innercustomgeneric
	}
	{\endinnercustomgeneric}
}

\newcustomtheorem{customthm}{Theorem}
\newcustomtheorem{customproposition}{Proposition}
\newcustomtheorem{customcorollary}{Corollary}
\newcustomtheorem{customlemma}{Lemma}

\DeclareMathOperator{\D}{\mathbb{D}}
\DeclareMathOperator{\N}{\mathbb{N}}
\DeclareMathOperator{\C}{\mathbb{C}}
\DeclareMathOperator{\R}{\mathbb{R}}
\DeclareMathOperator{\Ha}{\mathbb{H}}
\DeclareMathOperator{\Z}{\mathbb{Z}}
\DeclareMathOperator{\dist}{\mathrm{dist}}

\newcommand{\floor}[1]{\lfloor #1 \rfloor}

\begin{document}
	\title[On the Bergman number of hyperbolic domains]{On the Bergman number of hyperbolic domains}
	
	\author{Dimitrios Betsakos}  
	\address{Department of Mathematics, Aristotle University of Thessaloniki, 54124, Thessaloniki, Greece}
	\email{betsakos@math.auth.gr}  
	
	\author{Nikolaos Karamanlis}
	\address{Department of Mathematics, University of Thessaly, 3rd Km Old National Road Lamia-Athens, 35100, Lamia, Greece}
	\email{karamanlisn@gmail.com} 

	
	\date{}
	
	
	\begin{abstract} 
	We construct a domain $D$ in the plane whose Hardy and Bergman numbers satisfy $0<h(D)<b(D)<+\infty$. We also calculate the Bergman number for certain classes of domains having countable complement in the plane. Finally, we investigate some of the implications of our analysis in the theory of iteration of holomorphic self-maps of the unit disk. Our methods rely on estimates for the hyperbolic metric and a recent result \cite{CZRP} that connects the Bergman number with the hyperbolic metric.
	\end{abstract}
	
	\maketitle

\section{Introduction}
Although the theory of Bergman spaces has seen a remarkable development, concrete examples of holomorphic functions belonging to Bergman spaces are rare in the literature. Moreover, it remains difficult to decide whether a given holomorphic function belongs to a Bergman space or not. These problems
led to the recent introduction of the notion of the Bergman number of a planar domain, in analogy to the older notion of the Hardy number. 

\medskip

The Hardy space $H^p$ for $0<p<\infty$ is the set of holomorphic functions $f$ defined on the unit disk $\D$ satisfying
\[
\sup_{0<r<1}\int_{0}^{2\pi}|f(re^{i\theta})|^pd\theta<+\infty.
\]
For an introduction to the theory of Hardy spaces, see \cite{Dur}. For a domain $D\subset\C$, we denote by $\rm{Hol}(\D,D)$ the set of holomorphic functions $f:\D\to D$. The Hardy number of a function $f\in\rm{Hol}(\D,\C)$ is defined as
\[
h(f):=\sup\left(\{p>0:\ f\in H^p\}\cup\{0\}\right)\in [0,+\infty].
\]
The Hardy number, $h(D)$,  of a domain $D$ was introduced by Hansen in \cite{Hans1} and is defined as
\begin{align*}
h(D):&=\sup\left(\{p>0:\ {\rm Hol}(\D, D)\subset H^p\}\cup\{0\}\right)\\
    & =\inf\{h(f):\ f\in \rm{Hol}(\D, D)\} \in [0,+\infty].
\end{align*}

A classical question is to describe or estimate the Hardy number of a planar domain in terms of its geometry. Ess\'{e}n \cite{Ess} obtained estimates for the Hardy number of a hyperbolic domain in terms of certain harmonic measures and the logarithmic capacity. Later, Kim and Sugawa \cite{KimSug}, using the work of Ess\'{e}n, proved a general formula for the Hardy number in terms of harmonic measure. Recently, Cruz-Zamorano and the first author \cite{BZ} characterized the Hardy number of Greenian domains in terms of the Green function. The Hardy number has also been studied within certain geometric classes of domains, e.g., starlike domains \cite{Hans1}, Koenigs domains \cite{CZKR}. The relation of the Hardy number with the exit time of Brownian motion is explored in \cite{MM}. 

Closely related to Hardy spaces are the Bergman spaces. The weighted Bergman space with exponent $p>0$ and weight $\alpha>-1$ is denoted by $A_\alpha^p$ and consists of all holomorphic functions $f$ on $\D$ such that
\[
\int_{\D}|f(z)|^p\left(1-|z|\right)^{\alpha}dA(z)<+\infty,
\]
where $dA$ denotes Lebesgue area measure on $\D$. For an exposition regarding these spaces, see \cite{DurSch}.

In \cite{BKK}, Karafyllia and the authors  characterized the conformal maps which are in $A_\alpha^p$ in terms of certain hyperbolic distances.
The quantity
\[
b(f):=\sup\left(\bigg\{\frac{p}{\alpha+2}:\ p>0,\ \alpha>-1,\ f\in A_\alpha^p\bigg\}\cup\{0\}\right)\in [0,+\infty]
\]
is called the Bergman number of a holomorphic function $f$, and it was introduced by Karafyllia and  the second author   in \cite{KrfKrm} where it was proved that the Hardy and Bergman numbers coincide for conformal maps on $\D$. Later, Karafyllia, in \cite{Krf3}, introduced the Bergman number of a domain $D\subset\C$:
\[
b(D):= \inf\{b(f):\ f\in \rm{Hol}(\D, D)\} \in [0,+\infty].
\]
It is known (see \cite[Lemma 2.1]{KimSug}) that the Hardy number of a hyperbolic domain $D\subset\C$ is equal to the Hardy number of a universal covering map from $\D$ onto $D$. The same was shown to be true for the Bergman number of $D$ in \cite[Lemma 2.13]{BZ}. In the case that $D\neq\C$ is a simply connected domain, the above assertions are true if we replace the universal covering map with a Riemann map from $\D$ onto $D$. 

The work in \cite[Theorem 1.3]{KrfKrm} therefore implies that for a simply connected domain $D$, we have $h(D)=b(D)$. If $D=\C$, then both numbers are $0$; see \cite[Lemma 2.2]{Krf3}. In particular, for simply connected $D\neq\C$, we have
\[
b(D)=h(D)=\liminf_{R\to \infty} \frac{\rho_D(0,C_R)}{\log R},
\]
where $C_R=\partial D(0,R)\cap D$ and $\rho_D$ denotes the hyperbolic distance in $D$. As usual, $D(0,r)$ is the open disk centered at $0$ of radius $r$. Recently, Contreras, Cruz-Zamorano, and Rodr\'{i}guez-Piazza \cite[Theorem 1.2]{CZRP} have extended this identity for the Bergman number to multiply connected hyperbolic domains. Their proof relies on the link between the Bergman number and some related quantities called growth numbers. They have proved that if $D$ is an unbounded hyperbolic domain in $\C$ which contains $0$, then
\begin{equation}\label{bergmanid}
	b(D)=\liminf_{R\to +\infty}\frac{\rho_D(0,C_R)}{\log R}.
\end{equation}
The normalization $0\in D$ is not essential.

In general, for a multiply connected domain $D$, it holds that $h(D)\leq b(D)$. A proof can be found in \cite[Lemma 2.2]{Krf3}. Cruz-Zamorano  and the first author constructed examples of domains $D$ in \cite{BZ} for which it holds that $0<h(D)<b(D)=+\infty$. In addition, they proved that equality of the two numbers holds within a certain class of multiply connected domains. It is an open question to determine conditions on a domain $D$ which are necessary and sufficient for the equality $b(D)=h(D)$ to be true. The authors in \cite{BZ}, among other questions, also asked whether there exists a domain $D$ for which $0\leq h(D)<b(D)<+\infty$. The authors in \cite[Corollary 7.1]{CZRP} have given an affirmative answer to this question by providing examples of domains $D$ satisfying $0=h(D)<b(D)<+\infty$. These domains are of the form $\C\setminus\big(\{0\}\cup\{-\gamma^n:\ n\in\Z\}\big)$ for $\gamma>1$. We give an affirmative answer to this question as well by using different methods. Our result shows that it is possible to have a positive Hardy number.
\begin{theorem}\label{Theorem1}
There exists a planar domain $D$ satisfying $0<h(D)<b(D)<+\infty$.
\end{theorem}

In order to prove Theorem \ref{Theorem1}, we first prove the following result, which is of independent interest. The requirement $0\in D$ is not essential, but it is convenient. 
\begin{theorem}\label{Theorem2}
Let $D\subset\C$ be a domain containing 0. Suppose that there exists a sequence $z_j\in \C\setminus D$, $j\in\N$, satisfying the following:\\
(i) the sequence $r_j=|z_j|$ is strictly increasing to infinity,\\
(ii) there exists $M>0$ so that $\frac{r_{j+1}}{r_j}\leq M$, for all $j$.\\
Then there is a constant $C=C(M)$ which depends only on $M$ so that $b(D)\geq C>0$.
\end{theorem}
The proofs of Theorems \ref{Theorem1} and \ref{Theorem2} are in Section 3.

\medskip

Our next goal is to calculate the Bergman number for various domains. For $a\in (0,\pi)$, we let
\begin{equation}\label{sector}
	S_a=\{z:\ \Re z> |\Im z|\cot a\}
\end{equation}
be the sector of angle $2a$ with vertex at $0$, which is symmetric with respect to $\R$ and contains $1$. Let $E_a=(\Z+i\Z)\setminus\overline{S_a}$ be the part of the integer lattice which lies outside the closure of the sector $S_a$ and set $G_a=\C\setminus E_a$. For $a=\pi$, we set $G_\pi= \C\setminus\{-n:\ n\in\N\}$. We consider the family of domains $\mathcal{F}_1:=\{G_a:\ a\in (0,\pi]\}$. Note that by \cite[Theorem 2.4 (i)]{KimSug}, $h(G_a)=0$, for all $a\in (0,\pi]$. Using \eqref{bergmanid}, we can prove the following result.
\begin{theorem}\label{Theorem3}
The domains $G_a$ satisfy $b(G_a)=\frac{\pi}{2a}$, for all $a\in (0,\pi]$.
\end{theorem}

Finally, we provide a partial generalization of Theorem \ref{Theorem2}, which corresponds to the case $a=\pi$.

\begin{theorem}\label{Theorem4}
Let $r_n$, $n\in\N$, be a sequence of positive real numbers strictly increasing to infinity, satisfying $\sup_n |r_{n+1}-r_n|<+\infty$. Let $G=\C\setminus\{-r_j:\ j\in\N\}$. Then the domain $G$ satisfies $b(G)=1/2$.
\end{theorem}
We will denote by $\mathcal{F}_2$ the family of domains satisfying the hypotheses of Theorem \ref{Theorem4}. The proofs of Theorems \ref{Theorem3} and \ref{Theorem4} are in Section 4. 

\medskip

The domains of the class $\mathcal{F}_1$ are interesting from the point of view of the theory of iteration of holomorphic functions $\phi:\D\to\D$. 
In Section 5, we present a brief introduction to this theory. We prove estimates for the hyperbolic metric on domains in the class $\mathcal{F}_1$ and show that these estimates have consequences in iteration theory. 

\medskip

\textbf{Acknowledgments.} We are grateful to F. J. Cruz-Zamorano for carefully reading our paper and providing many helpful suggestions and corrections.


\section{Preliminaries}

The hyperbolic metric in the unit disk  $\D$ is defined as $\lambda_{\D} (z)|dz|$, where
\begin{equation}
\lambda_{\D} (z)=\frac{2}{1-|z|^2},\ z\in\D.
\end{equation}
The hyperbolic length of a piecewise $C^1$ curve $\gamma:[a,b]\to\D$ is the integral
\begin{equation}\label{length1}
\ell_{\D}(\gamma)=\int_{\gamma}\lambda_{\D}(z)|dz|.
\end{equation}
The hyperbolic distance between two points $a,b\in\D$ is defined as
\[
\rho_{\D}(a,b)=\inf_{\gamma}\ell_{\D}(\gamma),
\]
where the infimum is taken over all piecewise $C^1$ curves $\gamma$ joining $a$ to $b$ in $\D$. One can show that the hyperbolic metric is invariant under automorphisms of the unit disk, and this fact enables us to compute the hyperbolic distance explicitly in $\D$:
\begin{equation}\label{hypdisk}
\rho_{\D}(a,b)=\log\frac{1+|T(a,b)|}{1-|T(a,b)|},
\end{equation}
where $T(a,b)=\frac{a-b}{1-\bar{b}a}$, $a,b\in\D$.

A curve $\gamma: I\to\D$, where $I$ is some interval, is called a hyperbolic geodesic, or simply a geodesic, in $\D$ if for any $t_1\leq t_2$ in $I$ we have
\[
\rho_{\D}(\gamma(t_1),\gamma(t_2))=\ell_{\D}(\gamma\vert_{[t_1,t_2]}).
\]
The  geodesics in the unit disk are arcs of circles or lines which are perpendicular to the unit circle. 

If a domain $D\subset\C$ is hyperbolic, i.e., it has the property that $\C\setminus D$ contains at least two points, then there exists a holomorphic function $\pi:\D\to D$ onto $D$ which is a local biholomorphism, and it is unique up to composition with automorphisms of the unit disk. Such a function is called a universal covering map. In the case that $D$ is simply connected, a universal covering map is a usual Riemann map. Universal covering maps can be used to define the hyperbolic metric on any hyperbolic domain. If $D$ is hyperbolic, then there exists a unique metric $\lambda_D(w)|dw|$ on $D$ such that
\begin{equation}\label{density}
\lambda_D(\pi(z))|\pi'(z)|=\lambda_{\D}(z),\ z\in\D,
\end{equation}
for any choice of a universal covering map $\pi$. See \cite[Theorem 10.3]{BM}. This metric is called the hyperbolic metric in $D$. Just like in $\D$, we define the hyperbolic length of a piecewise $C^1$ curve $\gamma$ in $D$ by
\begin{equation}\label{length2}
\ell_D(\gamma)=\int_{\gamma}\lambda_D(w)|dw|
\end{equation}
and the hyperbolic distance between two points $a,b\in D$ by 
\begin{equation}\label{hdmultiply}
	\rho_D(a,b)=\inf_{\gamma}\ell_D(\gamma),
\end{equation}
where the infimum is taken over all piecewise $C^1$ curves joining $a$ to $b$ in $D$.

By a version of the Schwarz-Pick lemma (see \cite[Theorem 10.5]{BM}), if $\pi:\D\to D$ is a universal covering map, then
\begin{equation}\label{contraction}
	\rho_D(\pi(z),\pi(w))\leq \rho_{\D}(z,w), 
\end{equation}
for all $z,w\in\D$. Moreover, the same result implies that the hyperbolic distance satisfies a domain monotonicity property. Namely, if $D_1\subset D_2$ are hyperbolic domains, then $\lambda_{D_2}(z)\leq \lambda_{D_1}(z)$ for all $z\in D_1$, which in turn gives $\rho_{D_2}(z,w)\leq \rho_{D_1}(z,w)$, for all $z,w\in D_1$.

A curve $\sigma$ is said to be a (hyperbolic) geodesic in $D$ if $\sigma$ is a lift of a geodesic in $\D$, i.e., $\pi(\sigma)=\gamma$, for some  geodesic $\gamma$ in $\D$. By \eqref{density}, \eqref{length1},  \eqref{length2}, and a change of variable in the integral, it follows that $\ell_{\D}(\gamma)=\ell_D(\sigma)$ for every lift $\sigma$ of $\gamma$. When $D$ is multiply connected, it is known that geodesics in $D$ might not be minimizers for \eqref{hdmultiply}. See \cite[Proposition 1.9.30]{Abate}. For this reason, we will say that a geodesic $\gamma$ which realizes the infimum in \eqref{hdmultiply} is minimal in $D$. Furthermore, by \cite[Theorem 14.12 (iii), (iv)]{Mar} and the uniqueness of geodesics in $\D$, it follows that any two geodesics in a multiply connected domain are not homotopic.

We will need some estimates for the hyperbolic metric. For a domain $D$, let $\delta_D(z)=\dist(z,\partial D)$, $z\in D$, be the radius of the largest open disk centered at $z$ which is contained in $D$. If $D$ is hyperbolic, then by the domain monotonicity property we have
\begin{equation}\label{upperlambda}
	\lambda_D(z)\leq\frac{2}{\delta_D(z)}, \ z\in D.
\end{equation}
By Koebe's Theorem, when $D$ is simply connected, we also have
\[
\lambda_D(z)\geq\frac{1}{2\delta_D(z)}, \ z\in D.
\]
We will need the following lower bound for $\lambda_D$ proved by Beardon and Pommerenke in \cite{BP}, which applies to multiply connected hyperbolic domains $D$. For $z\in D$, consider the quantity
\begin{equation}\label{beta}
\beta_D(z):=\inf\Biggl\{\Bigg|\log\bigg|\frac{z-a}{b-a}\bigg|\Bigg|:\ a, b\in\partial D,\ |z-a|=\delta_D(z)\Biggr\}.
\end{equation}

\begin{customthm}{A}\label{TheoremA}
\textit{
Let $D$ be a hyperbolic domain in $\C$. There exist universal constants $C,k>0$ such that for all $z\in D$,
\begin{equation}
	\lambda_D(z)\geq \frac{C}{\delta_D(z)\left(k+\beta_D(z)\right)}.
\end{equation}
}
\end{customthm}

We will also need the following reflection principle due to Minda \cite[Theorem 3]{Minda}. We will denote by $R(A,L)$ the reflection of a set $A$ about the line $L$.

\begin{customthm}{B}\label{TheoremB}
	\textit{
		Let $D$ be a hyperbolic domain in $\C$. Suppose that $L$ is a line with $D\cap L\neq \emptyset$. Let $L^{-}$ and $L^{+}$ denote the two half planes determined by $L$ and set
		\[
		D^{-}=D\cap L^{-}\ \textrm{and}\ D^{+}=D\cap L^{+}.
		\] 
		If $D^{-}\neq\emptyset$ and $R(D^{-}, L)\subset D^{+}$, then for any piecewise $C^1$ curve $\gamma$ in $D^{-}$ we have that $\ell_D(\gamma)\geq \ell_D\left(R(\gamma, L)\right)$. Furthermore, equality holds if and only if $D$ is symmetric with respect to $L$.
	}
\end{customthm}

We will also use the behavior of the hyperbolic metric under circular symmetrization. Suppose that $D$ is a hyperbolic domain in $\C$. For  $r\geq 0$, set $T_r=\{z:|z|=r\}$ and denote by $|D\cap T_r|$ the angular Lebesgue measure of $D\cap T_r$. 
The circular symmetrization of $D$ with respect to the
positive semi-axis is
\begin{equation}
\label{circ}
D^*= \{z=re^{i\theta}: D\cap T_r \neq \varnothing,\, 2|\theta | <|D\cap T_r| \}
\cup \{-r:D\cap T_r=T_r\}.
\end{equation}

A. Weitsman \cite{Weitsman}, \cite[page 698]{Hay} proved that circular symmetrization reduces the hyperbolic metric in the following sense:

\begin{customthm}{C}\label{TheoremC}
	\textit{
		Let $D$ be a hyperbolic domain in $\C$. For every $z\in D$, 	$\lambda_D(z)\geq \lambda_{D^*}(|z|)$.
	}
\end{customthm}


\section{Proofs of Theorems \ref{Theorem1}, \ref{Theorem2}}
We begin with the proof of Theorem \ref{Theorem2}.

\noindent\textit{Proof of Theorem \ref{Theorem2}.}
Let $D\subset\C$ be a domain containing $0$ and suppose that $\{z_j,\ j\in\N\}$ is a sequence satisfying the two properties of Theorem \ref{Theorem2}. Let $T=\{z_j, j\in\N\}$ and set $\Omega=\C\setminus T$. Note that $D$ is a hyperbolic domain contained in $\Omega$ and thus by the definition of the Bergman number, $b(D)\geq b(\Omega)$. 

Let $R>0$ and denote by $C_R=\partial D(0,R)\cap \Omega$ the part of the circle centered at $0$ of radius $R$ which is contained in $\Omega$. Recall that $\rho_\Omega$ denotes the hyperbolic distance in $\Omega$.  Let $\Omega^*$ denote the circular symmetrization of $\Omega$ given by \eqref{circ}. An application of Theorem \ref{TheoremC} shows that $\rho_\Omega(0,C_R)\geq \rho_{\Omega^*}(0,C_R)$. In view of \eqref{bergmanid}, we infer that $b(\Omega)\geq b(\Omega^*)$. Therefore, it suffices to show the required estimate for $b(\Omega^*)$.

By an argument identical to the proof of \cite[Lemma 4.4]{CZ}, it follows that $\rho_{\Omega^*}(0,C_R)=\ell_{\Omega^*}\left([0,R]\right)$. Then by \eqref{length2},
\begin{equation}\label{P1Eq0}
	\rho_{\Omega^*}(0,C_R)=\ell_{\Omega^*}\left([0,R]\right)=\int_{[0,R]}\lambda_{\Omega^*}(z)|dz|.
\end{equation}

We shall now bound $\lambda_{\Omega^*}$. Since the Bergman number is left unchanged by applying translations to the domain, we may assume that $r_1=1$. Observe that for all $z\in [0,R]$, we have
\begin{equation}\label{P1dist}
\delta_{\Omega^*}(z)=|z+r_1|=|z+1|.
\end{equation}
Moreover, for each $z\in[0,R]$, we may find some $k\in\N$ with $r_k\leq 1+\delta_{\Omega^*}(z)<r_{k+1}$ because $r_j$ is increasing. Now \eqref{beta} gives
\begin{equation}\label{P1Eq1}
\beta_{\Omega^*}(z)\leq \left\lvert \log \frac{\delta_{\Omega^*}(z)}{r_k-1}\right\rvert\leq \log\frac{r_{k+1}-1}{r_k-1}.
\end{equation}
Since $\frac{r_{j+1}}{r_j}\leq M$ for all $j$, it follows that if $\delta_{\Omega^*}(z)$ is sufficiently large, then 
\[
\frac{r_{k+1}-1}{r_k-1}\leq \frac{Mr_k-1}{r_k-1}\leq M+\frac{M-1}{r_k-1}\leq 2M.
\]
Theorem \ref{TheoremA} together with the previous estimate and \eqref{P1Eq1} implies
\begin{equation}\label{P1Eq2}
\lambda_{\Omega^*}(z)\geq \frac{C}{(k+\beta_{\Omega^*}(z))}\frac{1}{\delta_{\Omega^*}(z)}\geq \frac{\tilde{C}}{\delta_{\Omega^*}(z)},
\end{equation}
for all $z\in [0,R]$ of sufficiently large modulus, say $|z|>R_0$, for some large $R_0>0$. Here we set $\tilde{C}=\frac{C}{k+\log (2M)}$.
Finally, using \eqref{bergmanid}, \eqref{P1Eq0}, \eqref{P1dist}, and \eqref{P1Eq2}, we calculate
\begin{align*}
b(\Omega^*)=\liminf_{R\to\infty}\frac{\rho_{\Omega^*}(0,C_R)}{\log R} &=
\liminf_{R\to\infty}\frac{1}{\log R}\int_{[0,R]}\lambda_{\Omega^*}(z)|dz|\\
&=\liminf_{R\to\infty}\frac{1}{\log R}\int_{[R_0,R]}\lambda_{\Omega^*}(z)|dz| \\
&\geq \tilde{C}\liminf_{R\to\infty}\frac{1}{\log R}\int_{[R_0,R]}\frac{|dz|}{\delta_{\Omega^*}(z)}\\
&= \tilde{C}\liminf_{R\to\infty}\frac{1}{\log R}\int_{R_0}^{R}\frac{dt}{t+1}\\
&= \tilde{C}.
\end{align*}
\qed

In \cite[Theorem 5.2]{CZKR}, it is proved that given $p>0$, there exists a domain $G$, containing $0$, such that $h(G)=p>0$. The domain $G$ is constructed inductively, and it is of the form
\[
G=\C\setminus \bigcup_{n\geq 1}\left(C_{R_n}\setminus \Gamma_n\right),
\]
where $R_n$ is a sequence tending to infinity and satisfying $R_1=2$ and $R_{n+1}>R_n^2$, for all $n$. Moreover, $C_{R_n}$ is the circle of radius $R_n$ centered at $0$ and $\Gamma_n=\bigg\{R_ne^{it}:\ |t|<\frac{\pi}{R_n^{2p}}\bigg\}$. 

For small $\epsilon>0$, we consider the domain 
\[
G_\epsilon=G\setminus \bigcup_{j\geq 1} A_j(\epsilon),
\]
where $A_j(\epsilon)=\{2^je^{it}:\ |t-\pi|<\epsilon\}$. See Figure \ref{Fig1}.

Let $x_n$, $y_n$ be two sequences of real numbers. Throughout the rest of the paper, we will write $x_n\lesssim y_n$ if there is a universal constant $C>0$ such that $x_n\leq Cy_n$, for all $n\in\N$. We will also write $x_n\sim y_n$ if $\lim_n \frac{x_n}{y_n}=1$.

\begin{figure}
	\includegraphics[width=0.7\linewidth]{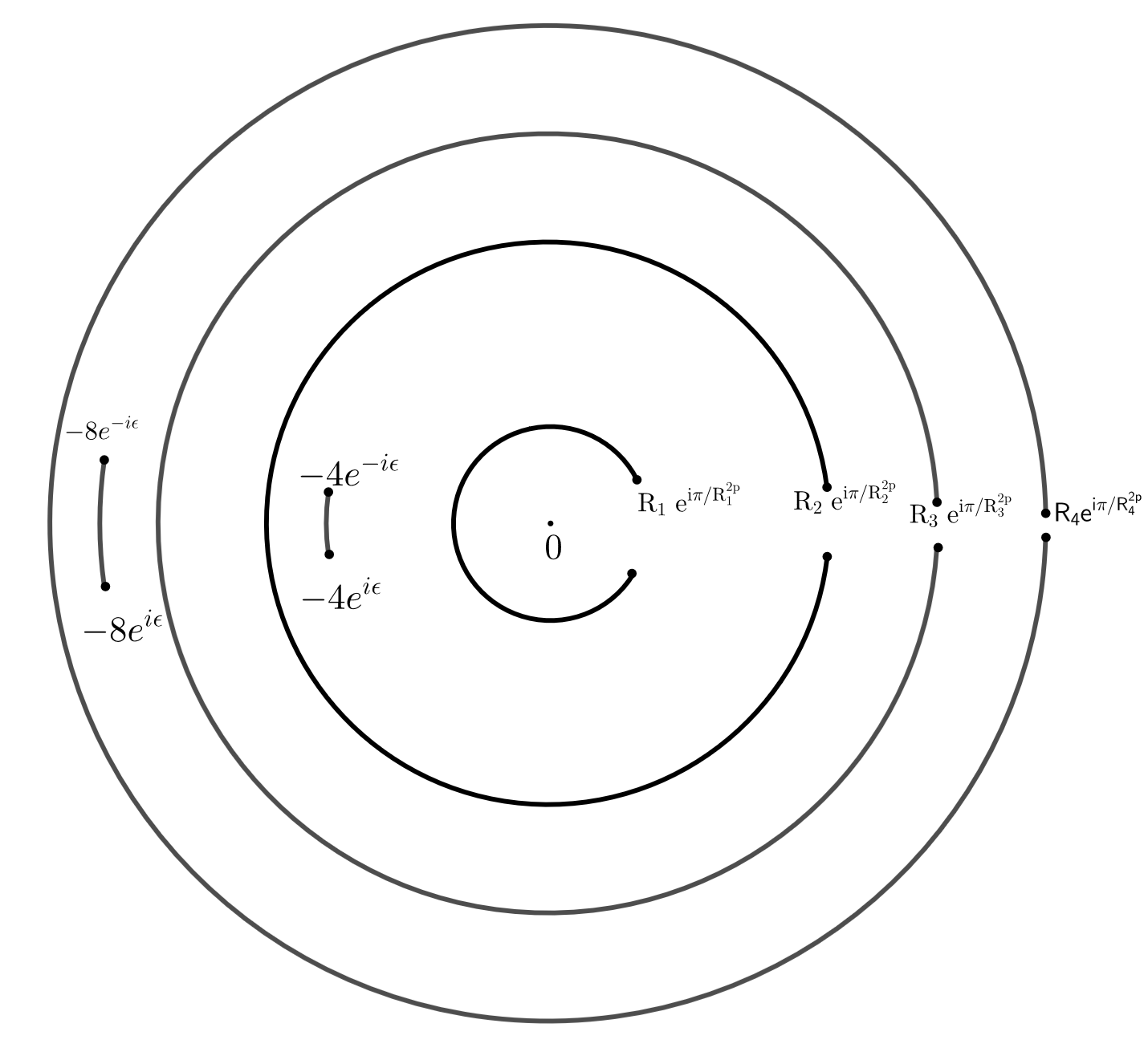}
	\caption{The domain $G_\epsilon$ for some $\epsilon>0$. The short circular arcs are $A_2(\epsilon)$ and $A_3(\epsilon)$.}
	\label{Fig1}	
\end{figure}

We need the following lemma.
\begin{lemma}\label{L1}
The domain $G_\epsilon$ satisfies $b(G_\epsilon)<\infty$, for all $\epsilon>0$ sufficiently small.
\end{lemma}
\begin{proof}
We first observe that 
\[
b(G_\epsilon)=\liminf_{R\to\infty}\frac{\rho_{G_\epsilon}(0,C_R)}{\log R}\leq \liminf_{n\to \infty}\frac{\rho_{G_\epsilon}(0,C_{R_n})}{\log R_n}\leq \liminf_{n\to \infty}\frac{\ell_{G_\epsilon}([0,R_n])}{\log R_n}.
\]
By the Stolz-Ces\`{a}ro theorem, the last term in this inequality is no larger than
\[
\limsup_{n\to \infty}\frac{\ell_{G_\epsilon}([0,R_{n+1}])-\ell_{G_\epsilon}([0,R_n])}{\log R_{n+1}-\log R_n}=\limsup_{n\to \infty}\frac{\ell_{G_\epsilon}([R_n,R_{n+1}])}{\log R_{n+1}-\log R_n}.
\]
Hence
\begin{equation}\label{L1Eq1}
	b(G_\epsilon)\leq \limsup_{n\to \infty}\frac{\ell_{G_\epsilon}([R_n,R_{n+1}])}{\log R_{n+1}-\log R_n},
\end{equation}
which shows that it suffices to bound $\ell_{G_\epsilon}([R_n,R_{n+1}])$. Note that by \eqref{upperlambda},
\[
\ell_{G_\epsilon}([R_n,R_{n+1}]) \lesssim \int_{R_n}^{R_{n+1}}\frac{dt}{\delta_{G_\epsilon}(t)}.
\]
An easy calculation shows that there exists a number $M_n\in (R_n,R_{n+1})$ which satisfies $\frac{M_n}{R_{n+1}}\to 1/2$, as $n\to\infty$, and is equidistant from the points $R_n e^{i\pi R_n^{-2p}}$ and $R_{n+1}e^{i\pi R_{n+1}^{-2p}}$. For $\epsilon>0$ sufficiently small, we then have
\[\delta_{G_\epsilon}(t)= \begin{cases} 
	\big| t-R_n e^{i\pi R_n^{-2p}}\big|, & t\in (R_n,M_n) \\
	\big| t-R_{n+1} e^{i\pi R_{n+1}^{-2p}}\big|, & t\in (M_n,R_{n+1}).
\end{cases}
\]
Set $c_n=\cos\pi R_n^{-2p}$ and $s_n=\sin\pi R_n^{-2p}$. We split the integral $\int_{R_n}^{R_{n+1}}\frac{dt}{\delta_{G_\epsilon}(t)}$ into three pieces over the intervals  $(R_n,M_n)$, $(M_n,R_{n+1}c_{n+1})$,  and \\ $(R_{n+1}c_{n+1},R_{n+1})$.

Using the estimates $1-c_n\sim \frac{\pi^2}{2} R_n^{-4p}$ and $s_n \sim \pi R_n^{-2p}$, we find
\begin{align*}
	\int_{M_n}^{R_{n+1}c_{n+1}}\frac{dt}{\delta_{G_\epsilon}(t)} &\lesssim \int_{M_n}^{R_{n+1}c_{n+1}}\frac{dt}{R_{n+1}c_{n+1}-t+R_{n+1}s_{n+1}}\\
	&=\log\frac{R_{n+1}c_{n+1}-M_n+R_{n+1}s_{n+1}}{R_{n+1}s_{n+1}} \\
	&=\log\left(1+\frac{c_{n+1}-M_n/R_{n+1}}{s_{n+1}}\right) \\
	&\lesssim \log R_{n+1}.
\end{align*}
Furthermore, 
\begin{align*}
\int_{R_{n+1}c_{n+1}}^{R_{n+1}}\frac{dt}{\delta_{G_\epsilon}(t)}&\lesssim 
\int_{R_{n+1}c_{n+1}}^{R_{n+1}}\frac{dt}{t-R_{n+1}c_{n+1}+R_{n+1}s_{n+1}}\\
&=\log\left(\frac{1-c_{n+1}}{s_{n+1}}+1\right)\to 0,
\end{align*}
as $n\to\infty$. In a similar fashion, we find
\begin{align*}
\int_{R_n}^{M_n}\frac{dt}{\delta_{G_\epsilon}(t)}\lesssim \int_{R_n}^{M_n}\frac{dt}{t-R_nc_n+R_ns_n}&=\log\frac{M_n-R_nc_n+R_ns_n}{R_n-R_nc_n+R_ns_n}\\
&\lesssim \log\frac{M_n-R_nc_n+R_n^{1-2p}}{R_n^{1-4p}+R_n^{1-2p}}\\
&=\log \frac{R_n^{2p}\left(M_n/R_n-c_n+R_n^{-2p}\right)}{R_n^{-2p}+1}\\
&\lesssim \log \left(R_n^{2p-1}M_n\right).
\end{align*}
Using the fact that $M_n<R_{n+1}$, we infer
\[
\int_{R_n}^{M_n}\frac{dt}{\delta_{G_\epsilon}(t)}\lesssim 2p\log R_n+\log\frac{R_{n+1}}{R_n}.
\]
Upon adding the estimates for the three integrals above, dividing by $\log R_{n+1}-\log R_n$, and using the fact that the latter quantity is no less than $\log R_n$ (since $R_{n+1}>R_n^2$), we find that 
\[
\limsup_{n\to\infty}\frac{\int_{R_n}^{R_{n+1}}\frac{dt}{\delta_{G_\epsilon}(t)}}{\log R_{n+1}-\log R_n}\leq 3+2p<\infty.
\]                            
The conclusion follows by \eqref{L1Eq1}.
\end{proof}

We can now finish the proof of Theorem \ref{Theorem1}.

\noindent\textit{Proof of Theorem \ref{Theorem1}.}
Consider the domain $D=\C\setminus\{2^k:\ k\in \N\}$. By Theorem \ref{Theorem2}, we have $b(D)>0$. Let $0<p<b(D)$. By Theorem 5.2 in \cite{CZKR}, there exists a domain $\Omega$ just as discussed before Lemma \ref{L1} so that $h(\Omega)=p<b(D)$. Set $G=D\cap\Omega$. Then $G$ is a domain which contains $0$ and satisfies $h(G)=h(\Omega)$. The last equality follows from \cite[Theorem 2.4(i)]{KimSug} because $\C\setminus D$ is polar. Hence,
\[
0<h(G)=h(\Omega)=p<b(D)\leq b(G).
\]
Finally, for $\epsilon>0$ small, since $G_\epsilon\subset G$, Lemma \ref{L1} gives
\[
b(G)\leq b(G_\epsilon)<\infty.
\]
\qed


\section{Proofs of Theorems \ref{Theorem3}, \ref{Theorem4}}

We recall that the family $\mathcal{F}_1$ contains the domains $G_a$, $a\in (0,\pi]$. For $a\in (0,\pi)$, we set 
\[
S_a=\{z:\ \Re z> |\Im z|\cot a\}
\]
and $E_a=(\Z+i\Z)\setminus\overline{S_a}$. Let $G_a=\C\setminus E_a$. For $a=\pi$, we let $G_\pi= \C\setminus\{-n:\ n\in\N\}$ and $S_\pi=\C\setminus (-\infty,-1]$, so that $S_a$ is the sector associated with the domain $G_a$ for all $a\in (0,\pi]$. See Figure \ref{Fig2}. We also defined the family  $\mathcal{F}_2$ which contains domains satisfying the hypotheses of Theorem \ref{Theorem4}. For $G\in \mathcal{F}_2$, we let $S=\C\setminus (-\infty,-r_1]$ be the sector associated with $G$. Finally, we let $\mathcal{F}:=\mathcal{F}_1\cup\mathcal{F}_2$.

\begin{figure}
	\includegraphics[width=0.5\linewidth]{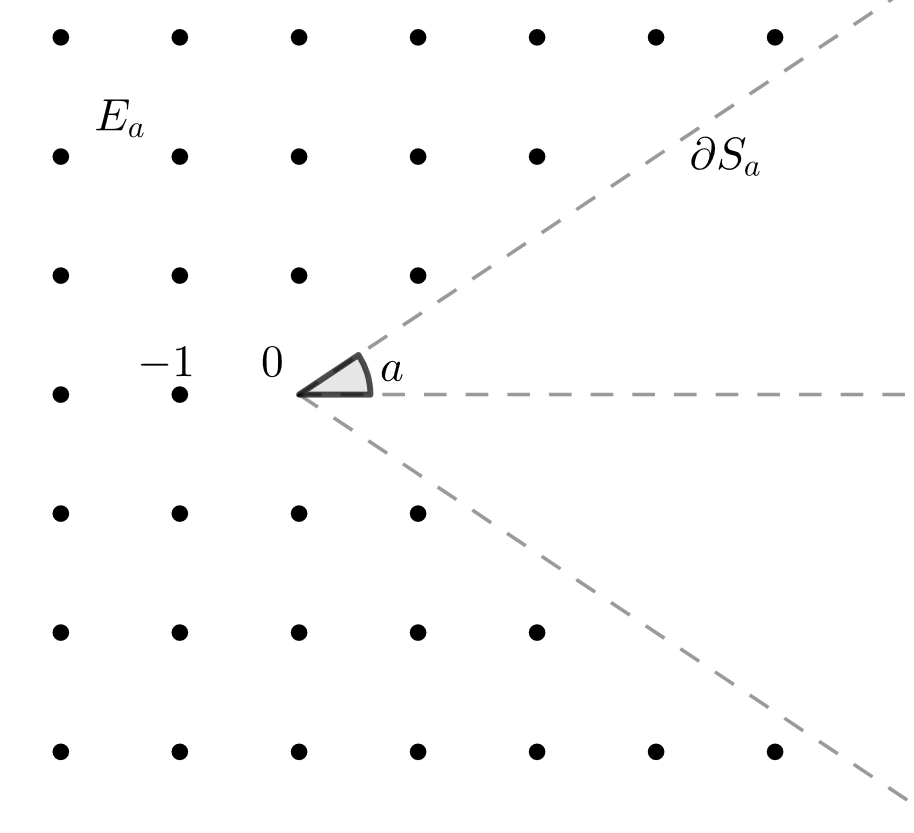}
	\caption{The domain $G_a$ for some $a\in (0,\pi/2)$. It is the complement of the set $E_a$.}
	\label{Fig2}	
\end{figure}

\begin{lemma}\label{L2}
Let $G\in \mathcal{F}_2$. Then 
\begin{equation}\label{dist}
\rho_G(t,\partial S)\to +\infty,
\end{equation}
as $t\to +\infty$, $t\in\R$.
\end{lemma}
\begin{proof}
Let $G=\C\setminus\{-r_j:\ j\in\N\}$ for some sequence $r_n$, $n\in\N$, of positive real numbers strictly increasing to infinity, and satisfying $M:=\sup_n |r_{n+1}-r_n|<+\infty$. We set $S=\C\setminus (-\infty,-r_1]$, and we have to prove \eqref{dist}. Let $0<\delta<M$. By the domain monotonicity property for the hyperbolic distance, we may assume that the sequence $\{r_j,\ j\in\N\}$ satisfies $\inf_n |r_{n+1}-r_n|\geq \delta$.

For each $t>0$, let $s_t\in (-\infty,-r_1)\setminus\{-r_j:\ j\in\N\}$ be  such that $\rho_G(t,\partial S)=\rho_G(t,s_t)$. Let $k=k(t)$ be an integer such that $-r_{k+1}<s_t<-r_k$. The domain monotonicity property then implies
\begin{equation}\label{L2Eq1}
\rho_G(t,s_t)\geq \rho_{G^*}(t,s_t),
\end{equation}
where $G^*=\C\setminus\{-r_{k+2},-r_{k+1}\}$. The conformal map $\frac{z+r_{k+1}}{r_{k+2}-r_{k+1}}$ transforms $G^*$ to $\C\setminus\{-1,0\}$ and the images of the points $s_t$ and $t$ satisfy 
\[
\frac{s_t+r_{k+1}}{r_{k+2}-r_{k+1}}<\frac{s_t+r_{k+1}}{\delta}\leq\frac{M}{\delta},
\]
and
\[
\frac{t+r_{k+1}}{r_{k+2}-r_{k+1}}\geq\frac{t}{M}.
\]
It follows that for $t>0$ sufficiently large,
\begin{equation}\label{L2Eq2}
\rho_{G^*}(t,s_t)=\rho_{\C\setminus\{-1,0\}}\left(\frac{s_t+r_{k+1}}{r_{k+2}-r_{k+1}},\frac{t+r_{k+1}}{r_{k+2}-r_{k+1}}\right)\geq\rho_{\C\setminus\{-1,0\}}\left(\frac{M}{\delta},\frac{t}{M}\right).
\end{equation}
 The last term on the right-hand side in \eqref{L2Eq2} tends to $+\infty$, as $t\to +\infty$ by a known estimate for $\lambda_{\C\setminus\{-1,0\}}$; see, for example, \cite[Eq. (1.5)]{BP}. The conclusion follows by \eqref{L2Eq1}.
\end{proof}

We also need a similar result for the family $\mathcal{F}_1$. 

\begin{lemma}\label{L3}
	Let $a\in (0,\pi]$ and consider the domain $G_a$ and the corresponding sector $S_a$. Then
	\begin{equation}\label{dist2}
		\rho_{G_a}(t,\partial S_a)\to +\infty,
	\end{equation}
	as $t\to +\infty$, $t\in\R$.
\end{lemma}
\begin{proof}
We first observe that it suffices to prove \eqref{dist2} for $a\in (0,\pi)$, since for $a=\pi$ the desired conclusion follows from Lemma \ref{L2} by choosing $r_j=j$, $j\in\N$. Fix $a\in (0,\pi)$ and let $G_a$ and $S_a$ be as above. For $t>0$, let $w_t\in\partial S_a$ be such that 
$\rho_{G_a}(t,\partial S_a)=\rho_{G_a}(t,w_t)$. Since the diameter of the integer lattice is smaller than $2$, we can find $a_t\in E_a$ so that $|w_t-a_t|<2$. Note that we also have $a_t-1\in E_a$. Let $G^*=\C\setminus\{a_t, a_t-1\}$ and observe that by the domain monotonicity property and a translation,
\begin{equation}\label{L3Eq1}
	\rho_{G_a}(t,w_t)\geq \rho_{G^*}(t,w_t)=\rho_{\C\setminus\{-1,0\}}(t-a_t,w_t-a_t).
\end{equation}
It follows from Theorem \ref{TheoremC} that 
\begin{equation}\label{L3Eq2}
\rho_{\C\setminus\{-1,0\}}(t-a_t,w_t-a_t)\geq \rho_{\C\setminus\{-1,0\}}\left(|t-a_t|,|w_t-a_t|\right).
\end{equation}
Since $a_t\in E_a$, a simple geometric consideration shows that $|t-a_t|>t\sin a$. This fact together with the inequality $|w_t-a_t|<2$ implies for $t$ sufficiently large that
\[
\rho_{\C\setminus\{-1,0\}}\left(|t-a_t|,|w_t-a_t|\right)>\rho_{\C\setminus\{-1,0\}}\left(2,t\sin a\right).
\]
Just like in the proof of the previous lemma, known estimates for $\lambda_{\C\setminus\{-1,0\}}$ show that 
\[
\rho_{\C\setminus\{-1,0\}}\left(2,t\sin a\right)\to +\infty,
\]
as $t\to +\infty$. Finally, by \eqref{L3Eq1} and \eqref{L3Eq2}, we have that $\rho_{G_a}(t,w_t)\to +\infty$, as $t\to +\infty$.
\end{proof}

We also need to identify certain minimal geodesics. 
\begin{lemma}\label{L4}
Let $G\in \mathcal{F}$. Then any real segment $[c,d]$, where $0<c<d$, is a minimal geodesic in $G$. 
\end{lemma}
\begin{proof}
If $G\in\mathcal{F}_2$ or $G=G_\pi$, then by a similar argument as the one given in \cite[Lemma 4.4]{CZ}, it follows that any real segment $[c,d]$, where $0<c<d$, is a minimal geodesic in $G$. We shall show that this is also true if the domain $G$ is some $G_a$ for $a\in (0,\pi)$. Fix $0<c<d$ and let $\gamma$ be a piecewise $C^1$ curve in $G$ joining $c$ to $d$ which is not homotopic to $[c,d]$. By the symmetry of $G$, we may assume that $\gamma \subset\overline{\Ha}$. For $k\in\Z$, we set $S_k=\{z\in\C:\ \Im z\in (k,k+1]\}$ and $L_k=\{z\in\C:\ \Im z=k\}$. We also let $L_k^{+}$ and $L_k^{-}$ be the upper and lower half planes determined by $L_k$, respectively. Let $M(\gamma)=\max\{k\in\Z:\ \gamma\cap S_k\neq \emptyset\}$ and observe that because $\gamma$ is not homotopic to $[c,d]$, we must have $M(\gamma)\geq 1$. 

Recall that $R(A,L_k)$ denotes the reflection of a set $A$ about the line $L_k$, and observe that for any $k\in\N$, we have that the reflection of $L_k^{+}$ about $L_k$ is contained in $L_k^{-}$. By Theorem \ref{TheoremB}, since $G$ is not symmetric with respect to $L_{M}$, where $M=M(\gamma)$, it follows that
\begin{equation}\label{minda}
	\ell_G\left(\gamma\cap L_M^{+}\right)>\ell_G\left(R(\gamma\cap L_M^{+},L_M)\right). 
\end{equation}
Let $\gamma_1$ be the curve defined by being equal to $\gamma(t)$ if $\gamma(t)\in \overline{L_M^{-}}$ and equal to $R(\gamma(t), L_M)$ if $\gamma(t)\in L_M^{+}$. Then \eqref{minda} gives $\ell_G(\gamma)>\ell_G(\gamma_1)$. We repeat this process until, after a finite number of reflections, we end up with a curve $\sigma$ joining $c$ to $d$ which is homotopic to $[c,d]$ and satisfies $\ell_G(\gamma)>\ell_G(\sigma)$. But then, since $\sigma$ is homotopic to $[c,d]$ and since the latter is a geodesic, we have $\ell_G(\sigma)\geq\ell_G([c,d])$. This means that $[c,d]$ is a minimal geodesic in $G$, i.e., $\rho_G(c,d)=\ell_G([c,d])$.
\end{proof}

We need one more auxiliary lemma which will be key to the proofs of Theorems \ref{Theorem3} and \ref{Theorem4}.
\begin{lemma}\label{L5}
Let $G\in\mathcal{F}$. Then
\begin{equation}\label{mainlimit}
\lim_{n\to +\infty}\frac{\rho_G(0,n)}{\log n} = b(S).
\end{equation}
\end{lemma}
\begin{proof}
Let $G\in\mathcal{F}$. For convenience, we let $S$ denote the sector associated with $G$, i.e., $S=\C\setminus (-\infty,-r_1]$ if $G\in\mathcal{F}_2$, while $S=S_a$ if $G\in\mathcal{F}_1$.
Let $\pi:\D\to G$ be the universal covering map satisfying $\pi(0)=0$ and $\pi'(0)>0$. By a known property of the universal covering map, \cite[page 81]{CDP}, there exists a univalent function $q:S\to\D$ such that $\pi(q(w))=w$, for all $w\in S$. Set $V=q(S)\subset\D$. The functional equation also shows that $\pi$ is univalent on $V$. It follows that $V$ is a simply connected domain in $\D$, and, because $S$ and $G$ are symmetric about $\R$, the same must hold for $V$. The above conditions imply that $q$ maps $(0,+\infty)$ onto $(0,1)$ in $\D$.

For $t>0$, let $J_t$ be the interval $[q(t),1)$ and consider for $R>0$ the sets 
\begin{equation}\label{L5Eq1}
	S_{\D}(t,R):=\{z\in\D:\ \rho_{\D}\left(z, J_t \right)<R \}.
\end{equation}
We claim that for each $R>0$, we may find $t_0>0$ so that $S_{\D}(t_0,R)\subset V$. If not, then  by \eqref{L5Eq1}, there is some fixed $R_0>0$, a sequence $t_n>0$ increasing to infinity, and points $z_n\in \D\setminus V$ satisfying
\[
\rho_{\D}\left(z_n, J_{t_n} \right)< R_0,
\]
for all $n\in\N$. For each $n$, let $x_n\in (0,+\infty)$ be such that
$\rho_{\D}(z_n,J_{t_n})=\rho_{\D}(z_n,q(x_n))$.
The hyperbolic geodesic in $\D$ joining $z_n$ to $q(x_n)\in J_{t_n}$ meets $\partial V$ at some point $\zeta_n$. Since $q$ is univalent, 
\begin{equation}\label{L5Eq2}
	R_0>\rho_{\D}\left(\zeta_n, J_{t_n} \right)=\rho_{\D}\left(\zeta_n, q(x_n) \right)=\rho_{\D}\left(q(w_n), q(x_n) \right),
\end{equation}
for some $w_n\in \partial S$. Since the universal covering map contracts hyperbolic distances, \eqref{contraction} and \eqref{L5Eq2} yield
\begin{equation}\label{L5Eq3}
	R_0>\rho_G\left(\pi(q(w_n)),\pi(q(x_n))\right)=\rho_G(w_n,x_n)\geq \rho_G(x_n,\partial S).
\end{equation}
Since $q(x_n)\in J_{t_n}$, we have $x_n>t_n$. Letting $n\to\infty$ in \eqref{L5Eq3} and using Lemmas \ref{L2} and \ref{L3} shows that this is a contradiction, and thus our claim holds. 

According to \cite[Definition 5.6]{CZ}, this implies that $V$ is internally tangent to $\D$ at $1$. Now applying \cite[Theorem 5.10 (c)]{CZ}, we conclude that for each $\epsilon>0$, there exists $t_0>0$ such that for any $t>t_0$ and any $n>t$ we have
\begin{equation}\label{L5Eq4}
\rho_V(q(t),q(n))\leq (1+\epsilon)\rho_{\D}(q(t),q(n)).
\end{equation}
Note that the real segment $[q(t),q(n)]$ is a lift of $[t,n]$ under $\pi$. Thus $\ell_{\D}([q(t),q(n)])=\ell_G([t,n])$, and by Lemma \ref{L4}, $\ell_G([t,n])=\rho_G(t,n)$. Since $q$ is symmetric with respect to $\R$, we have $\ell_{\D}([q(t),q(n)])=\rho_{\D}(q(t),q(n))$.

Therefore, we have proved that
\begin{equation}\label{L5Eq6}
\rho_{\D}(q(t),q(n))=\rho_G(t,n).
\end{equation} 
Moreover, the univalence of $\pi$ on $V$ gives
\begin{equation}\label{L5Eq7}
\rho_V(q(t),q(n))=\rho_S(t,n).
\end{equation}
By \eqref{L5Eq4}, \eqref{L5Eq6}, and \eqref{L5Eq7}, we infer that for each $\epsilon>0$, there exists $t_0>0$ so that
\begin{equation}\label{L5Eq8}
\rho_G(t,n)\geq\frac{1}{1+\epsilon}\rho_S(t,n),
\end{equation}
for all $n>t>t_0$. 

Note that for $n\in\N$,
\[
\rho_G(1,n)=\rho_G(1,t_0+1)+\rho_G(t_0+1,n).
\]
By \eqref{L5Eq8}, for $n$ large, we have
\[
\rho_G(1,n)\geq \rho_G(1,t_0+1)+\frac{1}{1+\epsilon}\rho_S(t_0+1,n),
\]
and a quick calculation shows that 
\[
\rho_S(t_0+1,n)=\log\frac{n^{\pi/2a}}{(t_0+1)^{\pi/2a}}.
\]
Putting the above together and setting $C_\epsilon:=\rho_G(1,t_0+1)-\frac{\pi}{2a(1+\epsilon)}\log(t_0+1)$, we obtain
\begin{equation}\label{L5Eq9}
	\rho_G(1,n)\geq C_\epsilon+\frac{\pi}{2a(1+\epsilon)}\log n,
\end{equation}
for all $n$ sufficiently large. Taking $\liminf$, as $n\to +\infty$, and then letting $\epsilon\to 0$ in \eqref{L5Eq9} gives
\[
	\liminf_{n\to +\infty}\frac{\rho_G(1,n)}{\log n} \geq \frac{\pi}{2a}.
\]
On the other hand, since $S\subset G$, the domain monotonicity property for the hyperbolic distance yields
\[
\limsup_{n\to +\infty}\frac{\rho_G(1,n)}{\log n}\leq \limsup_{n\to +\infty}\frac{\rho_S(1,n)}{\log n}=\limsup_{n\to +\infty}\frac{\log n^{\pi/2a}}{\log n}=\frac{\pi}{2a}.
\]
Hence
\[
\lim_{n\to +\infty}\frac{\rho_G(1,n)}{\log n} = \frac{\pi}{2a}.
\]
Using a conformal map, it is not hard to verify that $b(S)=\frac{\pi}{2a}$ when $G\in\mathcal{F}_1$ and $b(S)=\frac{1}{2}$ when $G\in\mathcal{F}_2$. By Lemma \ref{L4}, we may write $\rho_G(0,n)=\rho_G(0,1)+\rho_G(1,n)$, and \eqref{mainlimit} follows immediately. 
\end{proof}

\medskip

We are now ready to prove Theorems \ref{Theorem3} and \ref{Theorem4} .
\begin{proof}
Let $G\in\mathcal{F}$. According to \eqref{bergmanid}, it suffices to calculate
\[
\liminf_{R\to +\infty} \frac{\rho_G(0,C_R)}{\log R}.
\]
Since $S\subset G$, we have $b(G)\leq b(S)$. We will prove that the reverse inequality also holds. We begin by observing that \eqref{mainlimit} implies that 
\begin{equation}\label{limit2}
	\liminf_{R\to +\infty}\frac{\rho_G(0,R)}{\log R} \geq b(S).
\end{equation}
In order to see this, let $R_n\to +\infty$ be an arbitrary sequence and denote by $\floor{R_n}$ and $\{R_n\}$ the floor and fractional part of $R_n$, respectively. Then, for any $n\in\N$, we can write
\begin{equation}\label{Th2Eq1}
\frac{\rho_G(0,R_n)}{\log R_n}\geq\frac{\rho_G(0,\floor{R_n})}{\log \floor{R_n}}\cdot\frac{\log\floor{R_n}}{\log R_n}.
\end{equation}
Note that 
\[
\frac{\log\floor{R_n}}{\log R_n}=\frac{\log\floor{R_n}}{\log \floor{R_n}+\log\left(1+\{R_n\}/\floor{R_n}\right)}\to 1,
\]
as $n\to +\infty$. Taking liminf, as $n\to +\infty$, in \eqref{Th2Eq1} and using Lemma \ref{L5} gives \eqref{limit2}.

Next, we shall investigate the quantity $\rho_G(0,C_R)$ via a sequence of reflections across suitable lines. For each $R>0$, let $\gamma_R$ be a piecewise $C^1$ curve in $G$ joining $0$ to some point on $C_R$ so that $\rho_G(0,C_R)=\ell_G(\gamma_R)$. By the symmetry of $G$ with respect to $\R$ and Theorem \ref{TheoremB}, we may assume that $\gamma_R\subset\overline{\Ha}$. Note that by Theorem \ref{TheoremB} again, reflecting the subarcs of $\gamma_R$ which lie in $\{z:\ \Re z<0\}$ about the imaginary axis decreases their hyperbolic length. Thus we obtain a curve $\tilde{\gamma}_R$ which lies in $\{z: \Re z\geq 0,\ \Im z\geq 0\}$, joins $0$ to some point on $C_R$, and satisfies
\begin{equation}\label{Th2Eq2}
	\ell_G(\gamma_R)>\ell_G(\tilde{\gamma}_R).
\end{equation}
Let $L$ be the line $y=x$ and set $L^{-}=\{x+iy:\ y>x\}$ and $L^{+}=\{x+iy:\ y<x\}$. Observe that in all cases depending on the form of the domain $G$, we have $R(G^{-},L)\subset G^{+}$. Consequently, by Theorem \ref{TheoremB}, we have that reflecting the subarcs of $\tilde{\gamma}_R$ which lie in $L^{-}$ across $L$ decreases their hyperbolic length. Hence, we obtain a new curve $\sigma$ lying in $P:=\{x+iy:\ x\geq 0,\ 0\leq y\leq x\}$ which joins $0$ to some point $w_R\in C_R$ and satisfies
\begin{equation}\label{Th2Eq3}
	\ell_G(\tilde{\gamma}_R)>\ell_G(\sigma).
\end{equation}

\begin{figure}[t]
	\includegraphics[width=0.7\linewidth]{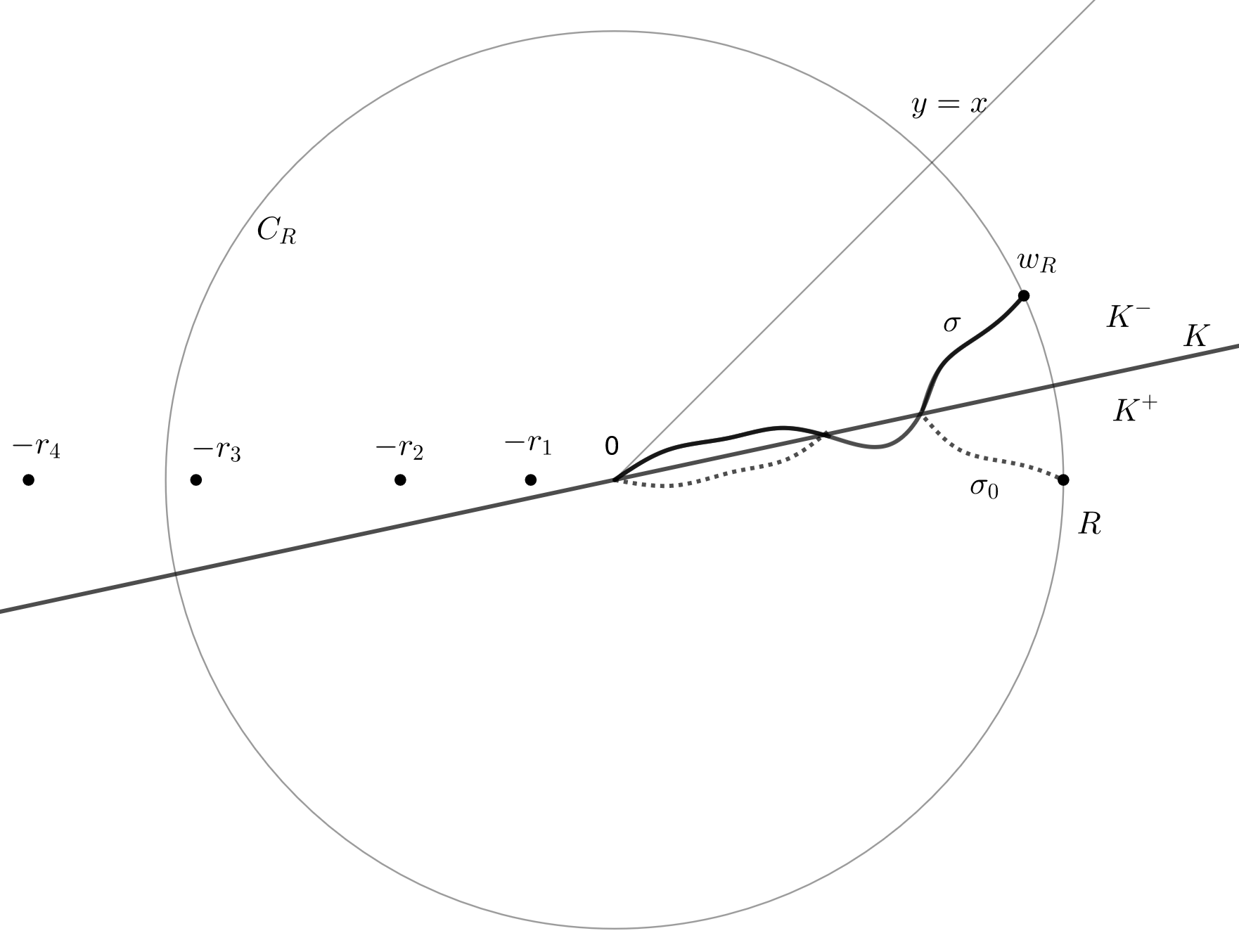}
	\caption{A domain which belongs to the family $\mathcal{F}_2$. The curve $\sigma$ lies in $P=\{x+iy:\ x\geq 0,\ 0\leq y\leq x\}$. The line $K$ is the perpendicular bisector of the segment $[w_R,R]$. The dotted arcs are the reflections of those arcs of $\sigma$ which lie in $K^{-}$. The curve $\sigma_0$ consists of the dotted arcs together with the arcs of $\sigma$ which lie in $\overline{K^{+}}$. It is contained in $\overline{K^+}\subset\overline{G^+}$.}
	\label{Fig3}	
\end{figure}

We now consider two cases depending on whether $G\in \mathcal{F}_1$ or $G\in \mathcal{F}_2$. \\
\textit{\underline{Case 1.}} Suppose $G\in \mathcal{F}_2$. The line $K$, which is the bisector of the chord of $C_R$ with endpoints $w_R$ and $R$ passes through $0$. Denote by $K^{-}$ the half plane which contains $\partial G$ and by $K^{+}$ the other one. Observe that the corresponding domains $G^{-}$ and $G^{+}$ with respect to $K$ satisfy $R(G^{-},K)\subset G^{+}$. See Figure \ref{Fig3}. By Theorem \ref{TheoremB}, we conclude that
reflection across the line $K$ decreases the hyperbolic length of any arc lying in $G^{-}$. Therefore, reflecting all such subarcs of $\sigma$ across $K$ yields a new curve $\sigma_0\subset \overline{G^{+}}$ which joins $0$ to $R$ and satisfies
\begin{equation}\label{Th2Eq4}
	\ell_G(\sigma)\geq \ell_G(\sigma_0).
\end{equation}
By Lemma \ref{L4}, since $\sigma_0$ and the segment $[0,R]$ are homotopic, we have that $\ell_G(\sigma_0)\geq \ell_G([0,R])=\rho_G(0,R)$. Then by \eqref{Th2Eq2}, \eqref{Th2Eq3}, and \eqref{Th2Eq4}, we obtain
\begin{equation}\label{Th2Eq5}
	\rho_G(0,C_R)\geq \rho_G(0,R).
\end{equation}
In view of \eqref{bergmanid} and \eqref{limit2}, \eqref{Th2Eq5} yields $b(G)\geq b(S)$. We now move on to the next case.\\
\textit{\underline{Case 2.}} Suppose now that $G\in\mathcal{F}_1$. Since $G_\pi\in\mathcal{F}_2$, we assume that $G=G_a$ for some $a\in (0,\pi)$. By \eqref{Th2Eq3}, we have established that $\rho_G(0,C_R)> \ell_G(\sigma)$ for some piecewise $C^1$ curve $\sigma\subset P$ which joins $0$ to a point $w_R\in C_R$. If $\floor{\Im w_R}>0$, we reflect the subarcs of $\sigma$ which lie in $\{x+iy: y>\floor{\Im w_R}\}$ across the line $y=\floor{\Im w_R}$. By Theorem \ref{TheoremB}, the resulting curve, $\sigma_1$, has smaller hyperbolic length than $\sigma$, and it joins $0$ to a point $w_{R}'$ which lies inside $D(0,R)$. If $\floor{\Im w_{R}'}>0$, we reflect again to decrease the hyperbolic length of the curve, and we repeat this process until we obtain a curve $\tilde{\sigma}$ which lies in $\{x+iy:\ y\in (0,1),\ x\geq 0\}$ and satisfies $\ell_G(\sigma)\geq \ell_G(\tilde{\sigma})$. This curve joins $0$ to a point $z_R:=x_R+iy_R$ which satisfies $x_R\geq \frac{R\sqrt{2}}{2}$ because the initial curve $\sigma$ was contained in $P$. Let $\tau_R$ be the vertical segment which joins the points $z_R$ and $\Re z_R$. See Figure \ref{Fig4}. By \eqref{upperlambda}, if $R$ is sufficiently large so that $z_R\in S_a$, we have
\begin{align*}
	\ell_G(\tau_R)=\int_{\tau_R}\lambda_G(w)|dw|&\lesssim \int_{\tau_R}\frac{|dw|}{\delta_G(w)}\\
	&\leq \int_{\tau_R}\frac{|dw|}{|w|\sin(a-\arg w)}\\
	&\lesssim \int_{\tau_R}\frac{|dw|}{|w|}\\
	& =\int_{0}^{y_R}\frac{dt}{|x_R+it|}\\
	&\leq \frac{1}{x_R}.
\end{align*}

\begin{figure}[t]
	\includegraphics[width=0.7\linewidth]{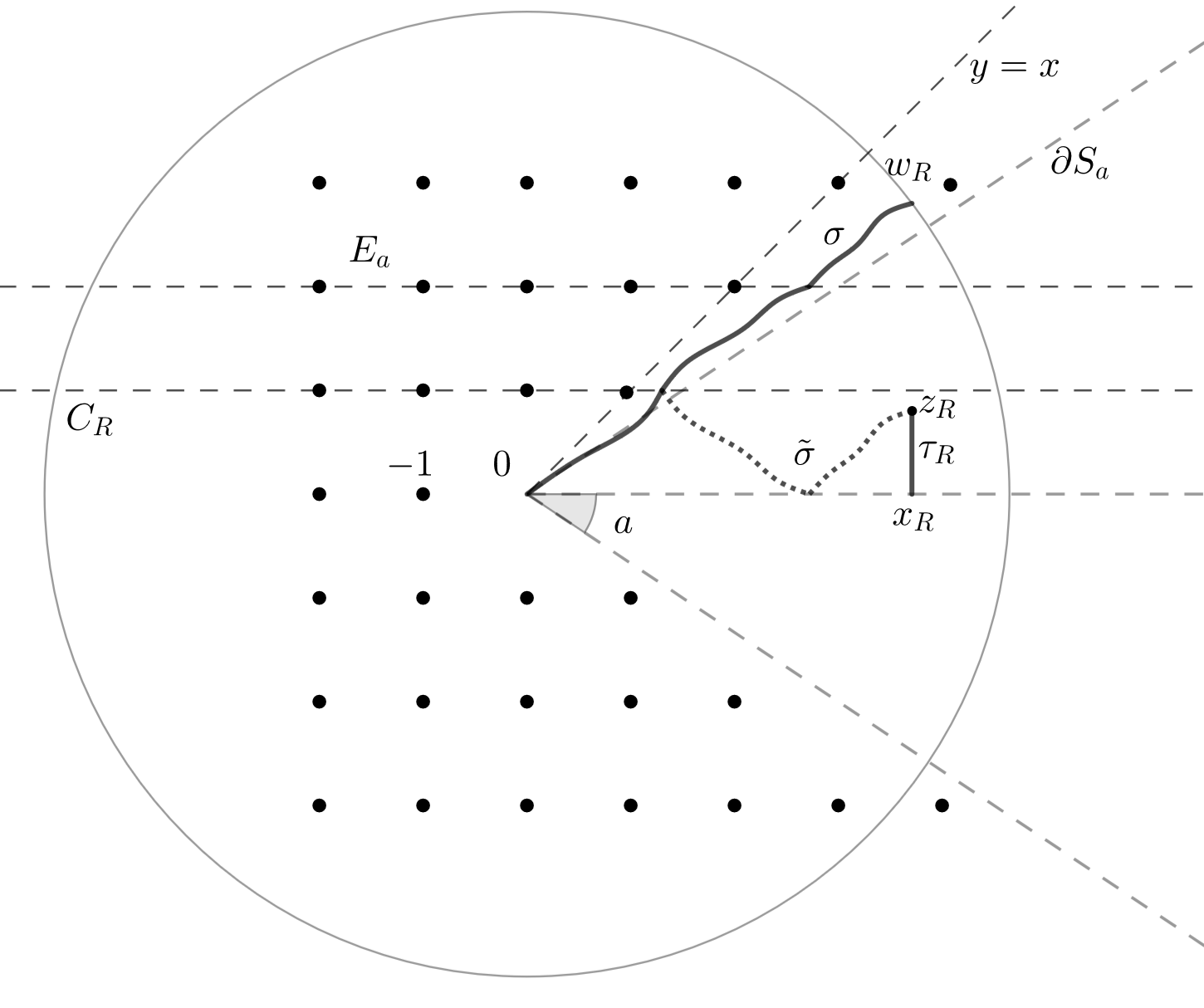}
	\caption{A domain $G_a$ of the family $\mathcal{F}_1$. The curve $\sigma$ lies in $P$. We reflect $\sigma\cap\{z:\ \Im z\in [2,3]\}$ about the line $y=2$. The arcs $\sigma\cap\{z: \Im z \leq 2\}$ together with this reflection are the trace of a new curve $\sigma_1$ which is not shown in the picture. We reflect $\sigma_1\cap\{z:\ \Im z\in [1,2]\}$ about the line $y=1$. The reflection is the dotted arc which lies in $\{z: \Im z\in [0,1]\}$. This arc together with $\sigma\cap\{z:\ \Im z\in [0,1]\}$ is the curve $\tilde{\sigma}$ which joins $0$ to the point $z_R$.}
	\label{Fig4}	
\end{figure}

Since $x_R\geq \frac{R\sqrt{2}}{2}$, it follows that $\ell_G(\tau_R)\to 0$, as $R\to +\infty$. Let $\tau$ be the curve we obtain by tracing $\tilde{\sigma}$ followed by $\tau_R$. This is a curve joining $0$ to $x_R$, and by Lemma \ref{L4}, we deduce that $\ell_G(\tau)>\rho_G(0,x_R)$. Putting the above together, we infer
\[
\rho_G(0,C_R)\geq \ell_G(\tilde{\sigma})= \ell_G(\tau)-\ell_G(\tau_R)>\rho_G\left(0,\frac{R\sqrt{2}}{2}\right)-\ell_G(\tau_R).
\]
Since $\ell_G(\tau_R)\to 0$, as $R\to +\infty$, it follows that
\begin{equation}\label{Th2Eq6}
\liminf_{R\to +\infty}\frac{\rho_G(0,C_R)}{\log R}\geq \liminf_{R\to +\infty}\frac{\rho_G\left(0,\frac{R\sqrt{2}}{2}\right)}{\log R}.
\end{equation}
Now write 
\[
\frac{\rho_G\left(0,R\sqrt{2}/2\right)}{\log R}=\frac{\rho_G\left(0,R\sqrt{2}/2\right)}{\log (R\sqrt{2}/2)}\cdot\frac{\log (R\sqrt{2}/2)}{\log R}
\]
and note that \eqref{limit2} and \eqref{Th2Eq6} yield
\begin{equation}\label{Th2Eq7}
\liminf_{R\to +\infty}\frac{\rho_G(0,C_R)}{\log R}\geq b(S).
\end{equation}
In view of \eqref{bergmanid}, \eqref{Th2Eq7} implies that $b(G)\geq b(S)$. We conclude that $b(G)=b(S)$. As we have noted before, using an appropriate conformal map shows that $b(S)=\frac{\pi}{2a}$ for $G\in \mathcal{F}_1$ and $b(S)=\frac{1}{2}$ for $G\in\mathcal{F}_2$. The desired results follow.
\end{proof}


\section{Applications to iteration theory}

In this section, we examine some of the implications of our results for iteration theory. Let $\phi$ be a self-map of the disk, i.e., $\phi\in \rm{Hol}(\D, \D)$.
Consider the sequence of functions $(\phi_n)$ defined by $\phi_0={\rm id}_{\D}$, $\phi_n=\phi\circ \phi_{n-1}$, $n=1,2,\dots$. It is 
 a consequence of the Schwarz Lemma that $\phi$ can have at most one fixed point in $\D$, and in this case this point is called the Denjoy-Wolff point of $\phi$. If no such point exists, then the Denjoy-Wolff theorem, \cite[Theorem 3.2.1]{Abate}, asserts that there exists a unique point $\tau\in\partial\D$ such that the iterates $\phi_n$ converge to $\tau$ locally uniformly in $\D$. In this case, $\tau$ is called the Denjoy-Wolff point of $\phi$, and it also holds that $\phi'(\tau)\in (0,1]$, where the derivative is an angular limit. This information can be used to classify the elements of $\rm{Hol}(\D, \D)$ as follows.\\
(i) elliptic: if $\tau\in\D$ and $\phi$ is not a hyperbolic rotation;\\
(ii) hyperbolic: if $\tau\in\partial\D$ and $\phi'(\tau)<1$;\\
(iii) parabolic: if $\tau\in\partial\D$ and $\phi'(\tau)=1$.

If $\phi$ is parabolic, it is known that either $\lim_{n} \rho_{\D} (z_n,z_{n+1})=0$, where $z_n=\phi_n(z_0)$ for some (equivalently for all) $z_0\in\D$ or $\lim_{n} \rho_{\D} (z_n,z_{n+1})>0$, where $z_n=\phi_n(z_0)$ for some (equivalently for all) $z_0\in\D$. In the first case, we say that $\phi$ is parabolic of zero hyperbolic step while in the latter case, we say that $\phi$ is parabolic of positive hyperbolic step. 

By Fatou's theorem, any $\phi\in\rm{Hol}(\D, \D)$ has non-tangential limits a.e. on $\partial\D$ with respect to the Lebesgue measure. One can therefore pose the natural question of whether $\phi$ satisfies the boundary version of the Denjoy-Wolff theorem, i.e., whether $\phi_n(\xi)\to \tau$, as $n\to +\infty$ for a.e. $\xi\in\partial\D$, where $\tau$ is the Denjoy-Wolff point of $\phi$. This question, along with a very nice exposition of the related theory, is studied in \cite{CDP}; see also \cite{BCD}. One of the main results discussed there is the following. Before we state it, recall that a self-map $\phi$ of the unit disk is called inner if $\phi$ has non-tangential limit of modulus $1$, a.e. on $\partial\D$. A sequence $\{z_n:\ n\in\N\}$ is called Blaschke if $\sum_n\left(1-|z_n|\right)<\infty$. A forward orbit of $\phi$ is a sequence $\phi_n(z_0)$ for some $z_0\in\D$.\\

\begin{customthm}{D}\label{TheoremD}
	\textit{ Let $\phi\in\rm{Hol}(\D, \D)$ be a self-map of $\D$ with Denjoy-Wolff point $\tau\in\overline{\D}$. The following assertions are true.\\
	{\rm (a)} \cite{BMS}, \cite{P-C} Suppose that $\phi$ is elliptic and not the identity. Then $\phi_n(\xi)\to \tau$ for a.e. $\xi\in\partial\D$ if and only if $\phi$ is not inner.\\
	{\rm (b)} \cite{BMS}, \cite{P-C} Suppose that $\phi$ is hyperbolic or parabolic of positive hyperbolic step. Then $\phi_n(\xi)\to \tau$ for a.e. $\xi\in\partial\D$.\\
	{\rm (c)} \cite{BMS}, \cite{P-C} Suppose that $\phi$ is an inner function which is also parabolic of zero hyperbolic step. Then $\phi_n(\xi)\to \tau$ for a.e. $\xi\in\partial\D$ if and only if some (equivalently all) forward orbit of $\phi$ is a Blaschke sequence.\\
	{\rm (d)} \cite{BCD} Suppose that $\phi$ is parabolic of zero hyperbolic step but is not an inner function. Then $\phi_n(\xi)\to \tau$ for a.e. $\xi\in\partial\D$.
	}
\end{customthm}

We now restrict our attention to domains $G\subsetneq\C$, symmetric about $\R$, containing the ray $[0,+\infty)$, and satisfying $G+1\subset G$. Such domains are necessarily hyperbolic. Therefore, there exists a unique universal covering map $\pi:\D\to G$ satisfying $\pi(0)=0$ and $\pi'(0)>0$. By the symmetry of $G$, we have that $\pi$ maps $[0,1)$ injectively onto $[0,+\infty)$. So, there is a point $\xi_1\in (0,1)$ such that $\pi(\xi_1)=1$. Moreover, by a lifting argument, there exists a unique holomorphic function $\phi:\D\to \D$ associated with $G$ such that $\phi(0)=\xi_1$ and 
\begin{equation}\label{funeq}
	\pi(\phi(z))=\pi(z)+1,\;\;\;z\in\D.
\end{equation}
  In the case that $\text{cap}(\C\setminus G)=0$, it follows that $\phi$ is inner. By \eqref{funeq}, we see that $\phi$ cannot be elliptic, and thus it is either hyperbolic or parabolic. It is also known that $\phi$ is parabolic of zero hyperbolic step if and only if $\dist(n,\partial G)\to +\infty$, as $n\to +\infty$. Furthermore, by the uniqueness of $\phi$ and the symmetry of $\pi$, we have that $\phi$ is symmetric with respect to $\R$. For these facts, see \cite[Theorems 8.1 and 8.2]{CDP} and references therein. 
  
  The above construction is of particular interest from the point of view of iteration theory because the orbits $(\phi_n(z))$, $z\in\D$, generated by $\phi$ correspond by $\pi$ to the translation sequences $(\pi(z)+n)$ in $G$. The dynamical properties of $\phi$  are hidden in the geometry of $G$, and it is an attractive problem to reveal them. 
  Regarding part (c) of Theorem \ref{TheoremD} above, the authors in \cite{CDP} show that within the class of inner self-maps of $\D$ which are parabolic of zero hyperbolic step, some maps satisfy the boundary version of the Denjoy-Wolff theorem, while others do not. They demonstrate this via the following examples.
  
  \medskip
  
\textbf{Example 1.}
 Let $E=\{m+in:\ m,n\in\Z,\ m<0\}$ and let $G=\C\setminus E$. Then $0\in G$ and $G+1\subset G$. Since $\lim_{n}\dist(n,\partial G)= +\infty$, it follows that the associated symmetric self-map $\phi$ of $\D$ is parabolic of zero hyperbolic step. Moreover, since $\text{cap}(E)=0$, we see that $\phi$ is inner. However, the authors proceed to show that $\sum_n\left(1-\phi_n(0)\right)=+\infty$. Hence, $\phi$ does not have Blaschke forward orbits, which by part (c) of Theorem \ref{TheoremD}, implies that $\phi$ does not satisfy the boundary version of the Denjoy-Wolff theorem.
 
 \medskip
 
 \textbf{Example 2.}
Let $E=\{m+in:\ m,n\in\Z,\ m<\alpha|n|\}$, for $\alpha\in (53,+\infty)$. Let $G=\C\setminus E$. Since $\lim_{n}\dist(n,\partial G)= +\infty$ we again have that the corresponding symmetric self-map $\phi$ is parabolic of zero hyperbolic step. Moreover, $\text{cap}(E)=0$, and thus $\phi$ is inner. However, as it turns out, in this case $\sum_n\left(1-\phi_n(0)\right)$ converges, which by part (c) of Theorem \ref{TheoremD}, implies that $\phi$ satisfies the boundary version of the Denjoy-Wolff theorem.

\medskip

These examples indicate that for $\alpha$ large ($\alpha\geq \pi/2$),
the corresponding $\phi$ does not satisfy the boundary version of the Denjoy-Wolff theorem. But if  $\alpha$ is small ($\alpha\leq \arctan(1/53)$), the opposite holds. 
Our results for the family $\mathcal{F}_1$ allow us to further study this class of examples. Observe that the domain $G$ in Example 1 is $G_{\pi/2}\in\mathcal{F}_1$ and the domains in Example 2 correspond to the domains $G_a\in\mathcal{F}_1$ for $a\in (0,\arctan(1/53))$.
The proposition below shows that the critical value of the angle $\alpha$ is
that of Example 1, namely $\pi/2$. 

\begin{proposition}
For each $a\in (0,\pi]$, let $\phi_a$ be the symmetric self-map of $\D$  associated with the domain $G_a$ of the family $\mathcal{F}_1$. Then $\phi_a$ is an inner, parabolic map of zero hyperbolic step, for all $a\in (0,\pi]$. Furthermore, $\phi_a$ satisfies the boundary version of the Denjoy-Wolff theorem if and only if $a\in(0,\pi/2)$.
\end{proposition}
\begin{proof}
Fix $a\in (0,\pi]$ and let $\phi=\phi_a$ be the corresponding self-map of $\D$ satisfying \eqref{funeq} and $\xi_1:=\phi(0)\in (0,1)$. Since $\C\setminus G_a$ has zero capacity,  $\phi$ is inner. Note that $\lim_{n}\dist(n,\partial G)=+\infty$, from which we infer that $\phi$ is a parabolic map of zero hyperbolic step. By the uniqueness of $\phi$ and the symmetry of $G$, it follows that $\xi_n:=\phi_n(0)\in (0,1)$ for all $n\in \N$. It follows from \eqref{funeq}  that $\pi(\xi_n)=\pi(\phi_n(0))=n$, for all $n\in\N$. 

 By \eqref{hypdisk} and an easy calculation in the unit disk,  we have for $n$ large enough
\begin{equation}\label{CorEq1}
e^{-\rho_{\D}(\xi_1,\xi_n)}=\frac{(1-\xi_n)(1-\xi_1)}{1-\xi_1\xi_n+\xi_n-\xi_1}\gtrsim 1-\xi_n.
\end{equation}
By \eqref{CorEq1} and \eqref{contraction}, we find
\begin{equation}\label{CorEq2}
\sum_n\left(1-\xi_n\right)\lesssim \sum_n e^{-\rho_{\D}(\xi_1,\xi_n)}\leq \sum_n e^{-\rho_G(\pi(\xi_1),\pi(\xi_n))}=\sum_n e^{-\rho_G(1, n)}.
\end{equation}
Moreover,  \eqref{L5Eq9} gives
\[
\sum_n e^{-\rho_G(1, n)}\lesssim \sum_n n^{-\frac{\pi}{2a(1+\epsilon)}},
\]
where the implied constant depends on $\epsilon$. The last estimate together with \eqref{CorEq2} shows that if $a\in (0,\pi/2)$, then the orbit $\xi_n$ is a Blaschke sequence. Hence, part (c) of Theorem \ref{TheoremD} implies that $\phi$ satisfies the boundary version of the Denjoy-Wolff theorem.

If $a\in [\pi/2,\pi]$, then by a direct calculation via a conformal map, it is not hard to see that
\begin{equation}\label{CorEq3}
\rho_{S_a}(1,n)=\log n^{\pi/2a}.
\end{equation}
Recall from the proof of Lemma \ref{L5} that there exists a univalent function $q:S_a\to\D$ such that $\pi(q(w))=w$, for all $w\in S_a$. We set $V=q(S_a)\subset\D$. Then $V$ is a simply connected domain in $\D$ which is symmetric with respect to $\R$ and $\pi$ is univalent on $V$. Furthermore, $q$ maps $(0,+\infty)$ onto $(0,1)$. Set $x_n=q^{-1}(\xi_n)\in (0,+\infty)$ and note that $x_n=\pi(q(x_n))=\pi(\xi_n)=n$, for all $n\in\N$. Then
\begin{align*}
\sum_n\left(1-\xi_n\right)\geq \sum_n e^{-\rho_{\D}(\xi_1,\xi_n)}&\geq \sum_n e^{-\rho_V(\xi_1,\xi_n)}\\
&=\sum_n e^{-\rho_{V}(q(x_1),q(x_n))}\\
&=\sum_n e^{-\rho_{S_a}(\pi(q(x_1)),\pi(q(x_n)))}\\
&=\sum_n e^{-\rho_{S_a}(x_1,x_n)}\\
&=\sum_n e^{-\rho_{S_a}(1,n)}\\
&=\sum_n n^{-\pi/2a}.
\end{align*}
The first inequality is evident from the identity in \eqref{CorEq1}. The second inequality holds by the domain monotonicity property because $V\subset\D$. The third-to-last equality follows because $\pi$ maps $V$ conformally onto $S_a$ and the second-to-last equality is a consequence of the functional equation $\pi(q(w))=w$, for all $w\in S_a$. The last equality follows from \eqref{CorEq3}. Observing that the last series diverges when $a\in [\pi/2,\pi]$ concludes the proof. 
\end{proof}


\end{document}